\documentclass[10 pt, a4paper]{article}
\usepackage[T1]{fontenc}
\usepackage[utf8]{inputenc}
\usepackage[english]{babel}
\usepackage{mathtools}
\usepackage{amssymb}
\usepackage{amsmath}
\usepackage{amsthm}
\usepackage{geometry}
\usepackage{midpage}
\usepackage{graphicx}
\usepackage{tabularx}
\usepackage{array}
\usepackage{comment}
\usepackage{mathrsfs}
\usepackage{tikz}
\usepackage{tikz-cd}
\usepackage[bottom]{footmisc}
\usepackage{latexsym}
\usepackage{emptypage}
\usepackage{fancyhdr}
\usepackage{regexpatch}
\usepackage[hidelinks]{hyperref}
\usepackage{filecontents}

\def\alg{\mathop{\mathrm{alg}}\nolimits}

\def\Hal{H^{\bullet}_{\alg}}
\def\Hom{\mathop{\mathrm{Hom}}\nolimits}
\def\Ext{\mathop{\mathrm{Ext}}\nolimits}
\def\irr{\mathop{\mathrm{irr}}\nolimits}
\def\Ulrich{\mathop{\mathrm{Ulrich}}\nolimits}
\def\Num{\mathop{\mathrm{Num}}\nolimits}
\def\Pic{\mathop{\mathrm{Pic}}\nolimits}
\def\ord{\mathop{\mathrm{ord}}\nolimits}
\def\tor{\mathop{\mathrm{tor}}\nolimits}
\def\gcd{\mathop{\mathrm{gcd}}\nolimits}

\makeatletter
\newtheorem*{rep@theorem}{\rep@title}
\newcommand{\newreptheorem}[2]{%
\newenvironment{rep#1}[1]{%
 \def\rep@title{#2 \ref{##1}}%
 \begin{rep@theorem}}%
 {\end{rep@theorem}}}
\makeatother

\newtheorem{Thm}{Theorem}[section]
\newreptheorem{Thm}{Theorem}
\newtheorem{Prop}[Thm]{Proposition}

\newtheorem{Lem}[Thm]{Lemma}
\newtheorem{Cor}[Thm]{Corollary}
\newreptheorem{Cor}{Corollary}

\newreptheorem{Con}{Conjecture}

\newtheorem*{theorem*}{Theorem}
\newtheorem*{lemma*}{Lemma}
\newtheorem*{proposition*}{Proposition}
\newtheorem*{conjecture*}{Conjecture}
\newtheorem*{corollary*}{Corollary}

\newtheorem{Thm-int}{Theorem}

\theoremstyle{definition}
\newtheorem{Def-s}[Thm]{Definition}
\newtheorem{Def}[Thm]{Definition}

\theoremstyle{remark}
\newtheorem{Rem}[Thm]{Remark}

\title{Vector bundles on bielliptic surfaces: Ulrich bundles and degree of irrationality}
\author{Edoardo Mason\footnote{\noindent Department of Mathematics, Stockholm University, Universitetsvägen 10 A, Stockholm, Sweden\\
Email: edoardo.mason@math.su.se\\
Data Availability Statement: This paper is a theoretical study and does not rely on any datasets.\\
Conflict of interest: The author has no conflict of interest to declare that is relevant to the
content of this paper.}}
\date{}

\begin{document}

\maketitle

\begin{abstract}
This paper deals with two problems about vector bundles on bielliptic surfaces. The first is to give a classification of Ulrich bundles on such surfaces $S$, which depends on the topological type of $S$. In doing so, we study the weak Brill-Noether property for moduli spaces of sheaves with isotropic Mukai vector. Adapting an idea of Moretti \cite{moretti}, we also interpret the problem of determining the degree of irrationality of bielliptic surfaces in terms of the existence of certain stable vector bundles of rank 2, completing the work of Yoshihara \cite{MR1810587}.
\end{abstract}

\section{Introduction}

Moduli spaces of stable sheaves on projective surfaces have been extensively studied in the last decades. Recently, Nuer \cite{nuer} investigated their non-emptiness and their birational geometry in the case of bielliptic surfaces $S$. In this paper, we use this knowledge about sheaves on bielliptic surfaces to study two problems: the classification of Chern characters of Ulrich bundles and the computation of the degree of irrationality of $S$.

\subsection*{Ulrich bundles}

Ulrich bundles made their first appearance in Commutative Algebra with \cite{ulr} and have become an active area of research in Algebraic Geometry since the influential paper \cite{esw}; we refer to \cite{beauville} for a detailed introduction. They are coherent sheaves $\mathscr{F}$ on a polarized projective variety $(X,H)$ such that $\mathscr{F}(-jH)$ has vanishing cohomology in all degrees for $1 \leq j \leq \dim X$. The interest in such sheaves originates from the fact that their existence has strong implications on the geometry of $X$. In particular, it implies that the \emph{cone of cohomology tables} of $(X,H)$
\[
\mathbb{Q}_{\geq 0}\left\langle\left(h^i(X,\mathscr{F}(jH))\right)_{0 \leq i \leq \dim X, j \in \mathbb{Z}}: \mathscr{F} \in \text{Coh}(X)\right\rangle \subseteq M_{\dim X + 1, \infty}(\mathbb{Q})
\]
is the same as for the projective space \cite{MR2810424}. Furthermore, it implies that the Cayley-Chow form of $(X,H)$ has a presentation as determinant of a matrix with linear entries in the Pl\"ucker coordinates. It is asked in \cite{esw} whether all projective varieties support Ulrich bundles and since then this has been proved in various cases; a non-exhaustive list includes curves \cite{esw}, ruled surfaces \cite{MR3680998}, K3 surfaces \cite{farkas}, abelian surfaces \cite{ulrab}, Fano threefolds \cite{MR4620673}, complete intersections \cite{MR1117634} and Grassmannians \cite{MR3302625}.

In the case of surfaces $S$ of Kodaira dimension zero, the existence of Ulrich bundles of rank 2 has been established (see e.g. \cite{beauville}), motivating the problem of classifying all Chern characters of Ulrich bundles. This problem can be seen as part of the more far-reaching question of determining the cohomology of the general stable sheaf on $S$ with a given Chern character. A starting point would be to determine for which Chern characters the following cohomology vanishing condition on the general stable sheaf is satisfied. 

\begin{Def} \label{wbn}
A moduli space of sheaves $M_H(\mathbf{v})$ satisfies the \emph{weak Brill-Noether property} if there exists a sheaf $\mathscr{F} \in M_H(\mathbf{v})$ such that $\mathscr{F}$ has at most one non-zero cohomology group.
\end{Def}

This question has been solved in the case of K3 surfaces of Picard rank one \cite{MR4593904} and of abelian surfaces \cite{MR4906400}, leading, as a corollary, to the classification of Ulrich bundles on such surfaces. In the case of Enriques surfaces, such classification depends on a lattice-theoretic conjecture formulated in \cite{MR3829179}. In this paper we address the problem of classifying Ulrich bundles on bielliptic surfaces. The starting point is the following result on the weak Brill-Noether property for line bundles, where we refer to Section 2 for the necessary notations.

\begin{Prop} \label{wbnlb}
\emph{[Lemma \ref{lemma1wbn}, \ref{lemma2wbn}]} Let $S$ be a bielliptic surface.
\begin{itemize}
    \item[$1)$] If $S$ is of type $1,2,3,5$, then the moduli space of line bundles of numerical class $aA_0 + bB_0$ satisfies the weak Brill-Noether property for all $a,b \in \mathbb{Z}$.
    \item[$2)$] If $S$ is of type $4,6,7$, then the weak Brill-Noether property fails exactly for the moduli spaces of line bundles of numerical class $bB_0$, with $\lvert b \rvert \geq \lambda_S$.
\end{itemize}
\end{Prop}

The problem of determining the weak Brill-Noether property for sheaves of higher rank remains open, but we give a positive result for isotropic Mukai vectors, see Proposition \ref{isotropic} for the precise statement.

The main classification result is the following. As we will see, the proof relies on the existence results for stable sheaves recently obtained by Nuer \cite{nuer} and reduces to Proposition \ref{wbnlb} and to a couple of key cases in rank 2 and 3 which are dealt with using both classical methods and derived category techniques.

\begin{Thm} \label{main}
Let $S$ be a bielliptic surface polarized by a generic ample divisor $H$ of numerical class $aA_0 +bB_0$, with $b \geq \lambda_S$, and assume that $H$ is not divisible in $\text{\emph{Num}}(S)$. Then:
\begin{itemize}
    \item[$1)$] If $S$ is of type $1,2,3,5$, then it carries Ulrich bundles of Mukai vectors
    \[
        \mathbf{v}^{\Ulrich}(r,k) := (r,kaA_0 + (3r-k)bB_0,2rab)
    \]
    for all $r\geq 1$, $r \leq k \leq 2r$.
    \item[$2)$] If $S$ is of type $4,6,7$, then it carries Ulrich bundles of every rank $r \geq 2$, but it does not carry Ulrich line bundles and not all $\mathbf{v}^{\Ulrich}(r,k)$, with $r\geq 2$, are Mukai vectors of Ulrich bundles.
\end{itemize}
\end{Thm}

\subsection*{Degree of irrationality}

\begin{Def} \label{irrdeg}
The \emph{degree of irrationality} $\irr(X)$ of an algebraic variety $X$ is the minimal degree of a generically finite rational map $X \dashrightarrow \mathbb{P}^{\dim X}$.
\end{Def}

Clearly, $\irr(X)$ is a birational invariant and, if $X$ is defined over a field $k$ and has rational function field $k(X)$, it coincides with the minimal possible degree $[k(X):k(x_1,\dots,x_{\dim X})]$, where $x_1,\dots,x_{\dim X}$ are algebraically independent elements of $k(X)$. For smooth projective curves, it coincides with the well-studied notion of gonality, while, already for surfaces, computing the degree of irrationality can be a very challenging problem. Since $\irr(X)$ has already been computed for very general hypersurfaces of high degree \cite{MR3705293} using positivity properties of the canonical bundle, the case of surfaces of Kodaira dimension zero is of particular interest and has indeed received much attention \cite{MR1337188}, \cite{MR1327053}, \cite{MR4039500}, \cite{MR4411875}, \cite{MR4891402}. Yoshihara \cite[Thm. 2.8]{MR1810587} computed $\irr(S)$ for all bielliptic surfaces except one case (type 6 in Table \ref{invariants}). In particular, he proved that $S$ has an involution such that the quotient is rational if and only if $S$ is of type 1 or 2, concluding that $\irr(S) = 2$ in these cases and $\irr(S) \geq 3$ otherwise. His strategy to prove the other inequality $\irr(S) \leq 3$ consists in showing the existence of a smooth curve of genus 3 in the linear system of an ample divisor on $S$, a method which fails for type 6. Here we prove the following result, which completes the work of Yoshihara.

\begin{Thm} \label{main2}
For all bielliptic surfaces $S$, we have $\irr(S) \leq 3$.
\end{Thm}

Our strategy provides a uniform treatment of the different types of bielliptic surfaces and relies on an adaptation of a method developed by Moretti \cite{moretti} for K3 surfaces.

Throughout the paper, we work over $\mathbb{C}$.

\section{Background material}

\subsection*{Ulrich bundles}

\begin{Def}
Let $X$ be a smooth projective variety. A Ulrich bundle on $X$ with respect to an ample divisor $H$ is a coherent sheaf $\mathscr{F}$ on $X$ such that
\[
H^{\bullet}(X,\mathscr{F}(-jH)) = 0 \quad \text{for $1 \leq j \leq \dim X$.}
\]
\end{Def}

We remark that, via the Auslander-Buchsbaum formula, this cohomology vanishing condition forces $\mathscr{F}$ to be locally free, justifying the use of the term \emph{bundle} in the definition. We refer to \cite[Thm. 2.3]{beauville} for equivalent characterizations of Ulrich bundles.

The starting point of our classification of Ulrich bundles on bielliptic surfaces is the fact that they are always semistable, hence their study can rely on the construction of moduli spaces of sheaves on surfaces developed by Gieseker and Maruyama. In the following we will work with $\mu$-semistability; let us nevertheless recall the chain of implications
\[
\mu\text{-stable} \Rightarrow \text{Gieseker stable} \Rightarrow \text{Gieseker semistable} \Rightarrow \mu\text{-semistable}.
\]

\begin{Thm}[{\cite[Thm. 2.9]{hartshorne}}] \label{semistable}
Let $E$ be a Ulrich bundle on a smooth projective variety $X$ with respect to $H$. Then:
\begin{itemize}
    \item[$1)$] $E$ is $H$-Gieseker semistable;
    \item[$2)$] if $E$ is $H$-Gieseker stable then it is also $\mu_H$-stable. 
\end{itemize}
\end{Thm}

\subsection*{Bielliptic surfaces}

According to the Kodaira-Enriques classification of smooth projective surfaces over $\mathbb{C}$, surfaces of Kodaira dimension zero are divided in 4 birational classes: K3 surfaces, abelian surfaces, Enriques surfaces and bielliptic surfaces. Among minimal surfaces of Kodaira dimension zero, bielliptic surfaces can be identified by their irregularity $q(S) = \dim H^1(S,\mathscr{O}_S) = 1$. 

If $S$ is a bielliptic surface, then its Albanese map is an elliptic fibration over an elliptic curve. Using the canonical bundle formula, one can deduce that the canonical bundle $\omega_S$ is algebraically trivial. Bielliptic surfaces have been classified by Bagnera-de Franchis: each bielliptic surface is the \'{e}tale quotient of a product of elliptic curves $A \times B$ by the action of a finite abelian group $G$ acting on $A$ by translations, so that $A/G$ is still an elliptic curve, and on $B$ in such a way that $B/G \cong \mathbb{P}^1$. Thus, $S = (A \times B) / G$ admits two elliptic fibrations induced by the $G$-equivariant projections of $A \times B$ onto its factors. In particular, $f:S \to A/G$ coincides with the Albanese morphism of $S$ and all its fibers are smooth and isomorphic to $B$, while the general fiber of $g:S \to B/G \cong \mathbb{P}^1$ is isomorphic to $A$. We adopt the usual abuse of notation of denoting by $A$ (resp. $B$) the class of the general fiber of $g:S \to \mathbb{P}^1$ (resp. $f:S \to A/G$) in $\Num(S) \cong H^2(S,\mathbb{Z})/\tor$. Then we define
\[
A_0 = \frac{1}{\ord(\omega_S)}A, \quad B_0 = \frac{1}{\lambda_S}B,
\]
where $\lambda_S = \frac{\lvert G \rvert}{\ord(\omega_S)}$, and by Serrano \cite{serrano} we have that $\Num(S) = \mathbb{Z}A_0 \oplus \mathbb{Z}B_0$. In this basis, the numerical group $\Num(S)$ is isomorphic to the hyperbolic plane with the standard symplectic form $\begin{psmallmatrix}
    0 &1\\
    1 &0
\end{psmallmatrix}$.

Depending on the group $G$, bielliptic surfaces fall into 7 families, whose invariants are summarized in Table \ref{invariants}, where $(m_1,\dots,m_s)$ denotes the multiplicities of the multiple fibers of $g$.

\begin{table}[ht]
\centering
\caption{Topological invariants of the seven families of bielliptic surfaces}
\begin{tabular}{c c c c c c}
\hline
Type & $G$ & $(m_1,\dots,m_s)$ & $H^2(S,\mathbb{Z})$ & $\ord(\omega_S)$\\
\hline
1 & $\mathbb{Z}/2\mathbb{Z}$ & $(2,2,2,2)$ & $\mathbb{Z}[\frac{1}{2}A]\oplus \mathbb{Z}[B]\oplus \mathbb{Z}/2\mathbb{Z}\oplus\mathbb{Z}/2\mathbb{Z}$ & 2\\
2 & $\mathbb{Z}/2\mathbb{Z}\times\mathbb{Z}/2\mathbb{Z}$ & $(2,2,2,2)$ & $\mathbb{Z}[\frac{1}{2}A]\oplus\mathbb{Z}[\frac{1}{2}B]\oplus\mathbb{Z}/2\mathbb{Z}$ & 2\\
3 &  $\mathbb{Z}/4\mathbb{Z}$ & $(2,4,4)$ & $\mathbb{Z}[\frac{1}{4}A]\oplus\mathbb{Z}[B]\oplus\mathbb{Z}/2\mathbb{Z}$ & 4\\
4 & $\mathbb{Z}/4\mathbb{Z}\times\mathbb{Z}/2\mathbb{Z}$ & $(2,4,4)$ & $\mathbb{Z}[\frac{1}{4}A]\oplus\mathbb{Z}[\frac{1}{2}B]$ & 4\\
5 & $\mathbb{Z}/3\mathbb{Z}$ & $(3,3,3)$ & $\mathbb{Z}[\frac{1}{3}A]\oplus\mathbb{Z}[B]\oplus\mathbb{Z}/3\mathbb{Z}$ & 3\\
6 & $\mathbb{Z}/3\mathbb{Z}\times\mathbb{Z}/3\mathbb{Z}$ & $(3,3,3)$ & $\mathbb{Z}[\frac{1}{3}A]\oplus\mathbb{Z}[\frac{1}{3}B]$ & 3\\
7 & $\mathbb{Z}/6\mathbb{Z}$ & $(2,3,6)$ & $\mathbb{Z}[\frac{1}{6}A]\oplus\mathbb{Z}[B]$ & 6\\
\hline
\end{tabular}
\label{invariants}
\end{table}

Since the canonical bundle $\omega_S$ has finite order in $\Pic(S)$, we can consider the associated \'{e}tale covering
\[
\pi:\tilde{X} = {\underline{\text{Spec}}}_{\mathscr{O}_S}(\mathscr{A}) \to S,
\]
where $\mathscr{A} = \bigoplus_{j=0}^{\ord(\omega_S)-1} \omega_S^{\otimes j}$ with its natural $\mathscr{O}_S$-algebra structure. It is easy to show that $\tilde{X}$ is an abelian surface sitting as an intermediate \'{e}tale quotient of $A \times B$ by a subgroup $\mathbb{Z}/\lambda_S\mathbb{Z} = H < G$. The abelian variety $\tilde{X}$ is called the \emph{canonical cover} of $S$. If it coincides with the product $A \times B$, that is, if $\lambda_S = 1$, the bielliptic surface $S$ is said to be \emph{split}. Clearly, $\tilde{X}$ carries two elliptic fibrations $\tilde{X} \to A/H$, $\tilde{X} \to B/H$ with fibers isomorphic to $B$ and $A$ respectively. Notice that, if we denote by $B_{\tilde{X}}$, $A_{\tilde{X}}$ the classes of these fibers in $\Num(\tilde{X})$, we have
\[
\pi^*(A_0) = A_{\tilde{X}}, \quad \pi^*(B_0)= \frac{\ord(\omega_S)}{\lambda_S}B_{\tilde{X}}.
\]

\begin{Lem}[{\cite[Lemma 2.4]{nuer}}] \label{lem:intermediate bielliptic order composite}
If $\ord(\omega_S)$ is composite, with proper divisor $d$, then there is a bielliptic surface $\tilde{S}$ sitting as an intermediate \'{e}tale cover between $S$ and $\tilde{X}$, $\pi':\tilde{S} \to S$,
such that $\ord(\omega_{\tilde{S}}) =\frac{\ord(\omega_S)}{d}$ and \[ \pi'^*(A_0)=\tilde{A}_0,\quad\pi'^*(B_0)=\frac{\ord(\omega_S)}{\ord(\omega_{\tilde{S}})}\tilde{B}_0=d\tilde{B}_0,
\] where $\tilde{A}_0,\tilde{B}_0$ are the natural generators of $\Num(\tilde{S})$.  
\end{Lem}

\begin{Lem}[{\cite[Lemma 2.5]{nuer}}] \label{lem:intermediate bielliptic lambda bigger than 1}
If $\lambda_S>1$, then there is a bielliptic surface $\tilde{S}$ sitting as an intermediate \'{e}tale cover between $S$ and $A\times B$, $\pi':\tilde{S} \to S$, such that $\lambda_{\tilde{S}}=1$, $\ord(\omega_{\tilde{S}})=\ord(\omega_S)$, and \[ \pi'^*(A_0)=\lambda_S\tilde{A}_0,\quad\pi'^*(B_0)=\tilde{B}_0,
\]
where $\tilde{A}_0,\tilde{B}_0$ are the natural generators of $\Num(\tilde{S})$.  
\end{Lem}

\subsection*{Stable sheaves on bielliptic surfaces}

The topological invariants of a coherent sheaf $\mathscr{F}$ on a smooth projective surface $X$ are encoded in its Mukai vector
\[
\mathbf{v}(\mathscr{F}) := \text{ch}(\mathscr{F})\sqrt{\text{td}(X)} \in H^\bullet(X,\mathbb{Q}).
\]
The definition of $\mathbf{v}$ extends to the bounded derived category of $X$ and factors through the numerical Grothendieck group $K(X)$. We denote by $\Hal(X,\mathbb{Z})$ the algebraic Mukai lattice of $X$, that is, the image of $K(X)$ under $\mathbf{v}$. In the cases we are interested in, $X$ is either a bielliptic surface or an abelian surface. In both cases $\mathbf{v}(\mathscr{F}) = \text{ch}(\mathscr{F})$ and
\[
\Hal(X,\mathbb{Z}) = H^0(X,\mathbb{Z}) \oplus \Num(X) \oplus H^4(X,\mathbb{Z}).
\]
The algebraic Mukai lattice carries a bilinear pairing, called Mukai pairing, which is defined as $\langle (r,c,s),(r',c',s')\rangle = c\cdot c' - rs' - r's$.
By the Hirzebruch-Riemann-Roch theorem, we have
\[
\langle \mathbf{v}(\mathscr{F}),\mathbf{v}(\mathscr{E})\rangle =-\chi(\mathscr{F},\mathscr{E}).
\]

Now let $S$ be a bielliptic surface. For any Mukai vector $\mathbf{v} \in \Hal(S,\mathbb{Z})$, there is a locally finite set of walls in the ample cone corresponding to polarizations $H$ for which a $\mu_H$-semistable sheaf of Mukai vector $\mathbf{v}$ contains a subsheaf of strictly smaller positive rank and same $H$-slope. A polarization $H$ is said to be \emph{generic} with respect to $\mathbf{v}$ if it does not lie in any wall for $\mathbf{v}$. The Bogomolov inequality on Chern classes of semistable sheaves implies that a necessary condition for the existence of a $\mu_H$-semistable sheaf of Mukai vector $\mathbf{v}$ is $\mathbf{v}^2 \geq 0$. Nuer \cite[Thm. 1.1]{nuer} proved that for generic polarizations $H$ this condition is also sufficient. We denote by $M_H(\mathbf{v})$ the moduli space parametrizing flat families of $\mu_H$-semistable sheaves on $S$ of Mukai vector $\mathbf{v}$. It is a projective variety containing the moduli space $M_H^{\mu s}(\mathbf{v})$ of $\mu_H$-stable sheaves of Mukai vector $\mathbf{v}$ as open (possibly empty) subset.

\begin{Def}
A Mukai vector $\mathbf{v}$ is said to be primitive if it is not divisible in $\Hal(S,\mathbb{Z})$.
For a primitive Mukai vector $\mathbf{v}=(r,aA_0+bB_0,s)$ we set $c_1(\mathbf{v}) := aA_0+bB_0$ and we define
\[
l(\mathbf{v}) = \gcd(r,\pi^*c_1(\mathbf{v}),\text{ord}(\omega_S)s) := \gcd(r,a,\frac{\text{ord}(\omega_S)}{\lambda_S}b,\text{ord}(\omega_S)s),
\]
so that $\frac{\pi^*\mathbf{v}}{l(\mathbf{v})}$ is primitive in the algebraic Mukai lattice of $\tilde{X}$.
\end{Def}

We remark that since our definition of the Mukai vector does not keep track of the torsion part of $H^2(S,\mathbb{Z})$, we have to expect $M_H(\mathbf{v})$ to be in general not connected for bielliptic surfaces $S$ of type $1,2,3,5$. Indeed, $M_H(\mathbf{v})$ decomposes as
\[
M_H(\mathbf{v}) = \bigsqcup_{[L]_{\alg}} M_H(\mathbf{v},L),
\]
where each $M_H(\mathbf{v},L)$ contains the sheaves $\mathscr{F}$ in $M_H(\mathbf{v})$ such that $\det(\mathscr{F})$ is algebraically equivalent to $L$ and the union is over all algebraic equivalence classes of line bundles of numerical class $c_1(\mathbf{v}) \in \Num(S)$, which is a set of cardinality equal to the cardinality of $H^2(S,\mathbb{Z})_{\tor}$. In particular, if $\mathbf{v} = (r,c,s)$ is such that $\mathbf{v}^2 \geq 0$ and $r$ is coprime with the orders of all elements in $H^2(S,\mathbb{Z})_{\tor}$ then all pieces of the above disjoint union are non-empty. The following lemma shows that, under suitable assumptions, this decomposition is exactly the decomposition of $M_H(\mathbf{v})$ into irreducible components.

\begin{Lem} \label{irredcomp}
Let $S$ be a bielliptic surface, let $D \in \Num(S)$ and $s$, $r \geq 2$ be integers such that 
\[
\gcd(r,D) = \gcd(r,\ord(\omega_S)) = 1, \quad D^2 -2rs >0.
\]
Then the number of irreducible components of $M_H(r,D,s)$ is equal to the cardinality of $H^2(S,\mathbb{Z})_{\tor}$.
\end{Lem}

\begin{proof}
    We claim that $M_H^{\mu s}(r,D,s)$ is smooth and that the number of its connected components is equal to the cardinality of $H^2(S,\mathbb{Z})_{\tor}$, from which the result follows, since each irreducible component of $M_H(r,D,s)$ contains a $\mu_H$-stable sheaf by \cite[Thm. 1.1]{nuer}.
    
    Let $E$ be a $\mu_H$-stable sheaf in $M_H(r,D,s)$, then
    \[
    \Ext^2(E,E) \cong \Hom(E,E\otimes \omega_S)^{*}.
    \]
    Here $E$ and $E\otimes \omega_S$ are $\mu_H$-stable sheaves of the same slope, hence $\Hom(E,E\otimes \omega_S)$ is non-zero if and only if $E \cong E\otimes \omega_S$. However, this implies that $E$ is the pushforward of a sheaf on the canonical cover $\tilde{X}$, which is not compatible with the assumption $\gcd(r,\ord(\omega_S)) = 1$. Thus $\text{Ext}^2(E,E) = 0$, hence $M_H^{\mu s}(r,D,s)$ is smooth at $[E]$.

    Now recall from \cite[Lemma 10.1]{nuer} that for any Mukai vector $\mathbf{v}$ such that $\mathbf{v}^2 \geq 0$, the virtual Hodge polynomial 
    \[
    e(M_H(\mathbf{v})) = \sum_{p,q} e^{p,q}(M_H(\mathbf{v}))x^py^q,
    \]
    where $e^{p,q}(M_H(\mathbf{v})) = \sum_k (-1)^k h^{p,q}(H^k(M_H(\mathbf{v}),\mathbb{Z}))$ for the Hodge numbers $h^{p,q}$ of the natural mixed Hodge structure on the cohomology of $M_H(\mathbf{v})$, is invariant under derived autoequivalences of $S$. 

    Notice that if we write $D = xA_0 + yB_0$, then the fiber degree of $\mathbf{v} = (r,D,s)$ is $\frac{x\lvert G \rvert}{\ord(\omega_S)}$ along $f:S \to A/G$ and $\ord(\omega_S)y$ along $g:S \to \mathbb{P}^1$, therefore our coprimality assumptions imply that $r$ is coprime to at least one of the two fiber degrees. Then we can use the derived autoequivalences (relative to $f$ or $g$) constructed by Bridgeland \cite{brid} to conclude that $M_H(r,D,s)$ has the same Hodge polynomial as that of $M_H(1,0,-\mathbf{v}^2/2)$. Indeed, if $\gcd\left(r,\frac{x\lvert G \rvert}{\ord(\omega_S)}\right) = 1$, then, letting $d := \frac{x\lvert G \rvert}{\ord(\omega_S)}$, there are unique integers $a,b$ such that $br-ad = 1$ and $1 \leq a < r$. By \cite[Thm. 5.3]{brid}, together with the fact that bielliptic surfaces have no non-trivial Fourier-Mukai partners \cite[Prop. 6.2]{MR1827500}, to the matrix $\begin{psmallmatrix}
    -b &a\\
    d &-r
\end{psmallmatrix} \in \text{SL}_2(\mathbb{Z})$ one can associate a derived autoequivalence $\Psi:D^b(S) \to D^b(S)$ such that, for $E \in M_H(r,D,s)$, $\Psi(E)$ has rank 1 and fiber degree zero along $f$. Twisting with a line bundle on $S$, we may assume that $\det(\Psi(E))$ is numerically trivial. But then, as the Mukai pairing is preserved by derived autoequivalences, $\Psi(E) \in M_H(1,0,-\mathbf{v}^2/2)$. If $\gcd(r,\ord(\omega_S)y) = 1$, the argument is analogous. 
    
Now, notice that
    \[
    M_H(1,0,-\mathbf{v}^2/2) \cong \ker(\Pic(S) \to \Num(S)) \times \text{Hilb}^{\frac{\mathbf{v}^2}{2}}(S),
    \]
    so its number of connected components is equal to the cardinality of $H^2(S,\mathbb{Z})_{\tor}$. The claim then follows by noticing that the number of connected components of a smooth variety is the rank of the singular cohomology group in degree zero, thus can be read from the constant term of the Hodge polynomial.
\end{proof}

\section{Weak Brill-Noether for line bundles on bielliptic surfaces}

In this section we determine which components of the Picard scheme of a bielliptic surface $S$ satisfy the weak Brill-Noether property. In particular, we will see that this property depends on the existence of non-zero torsion elements in $H^2(S,\mathbb{Z})$ only for divisors of self intersection zero.

\begin{Prop}[{\cite[Lemma 1.3]{serrano}}] \label{p0}
Let $D$ be a divisor on $S$ of numerical class $aA_0 + bB_0$, with $a,b \in \mathbb{Z}$. Then:
\begin{itemize}
    \item[$1)$] $\chi(S,\mathscr{O}_S(D)) = ab$;
    \item[$2)$] $D$ is ample if and only if $a>0$, $b>0$ and in this case $H^i(S,\mathscr{O}_S(D)) = 0$ for $i\geq 1$;
    \item[$3)$] If $H^0(S,\mathscr{O}_S(D)) \neq 0$, then $a\geq 0$, $b\geq 0$.
\end{itemize}
\end{Prop}

In the formalism of Definition \ref{wbn} we thus obtain the following result.

\begin{Cor} \label{c1}
For all bielliptic surfaces $S$, the moduli space of line bundles of numerical class $D = aA_0 + bB_0$ satisfies the weak Brill-Noether property if $ab\neq0$ or $a=b= 0$.
\end{Cor}

\begin{proof}
In the case $ab\neq 0$, the claim follows easily from Proposition \ref{p0} using Serre duality and the fact that $\omega_S \in \Pic^0(S)$. For $a=b=0$, notice that the moduli space of numerically trivial line bundles on $S$ contains $\Pic^0(S)$ and all non-trivial line bundles in $\Pic^0(S)$ have vanishing $H^0$. But then $H^2(S,\alpha) = 0$ for all $\alpha \in \Pic^0(S) \smallsetminus \{\omega_S\}$ by Serre duality, proving the weak Brill-Noether property, which is equivalent to the general line bundle in $\text{Pic}^0(S)$ having no non-zero cohomology group, since $\chi(S,\mathscr{O}_S) = 0$.
\end{proof}

\begin{Rem} \label{numtriv}
Recall that if $H^2(S,\mathbb{Z})_{\tor}= 0$, then algebraic and numerical equivalence of divisors are equivalent, so that the moduli space of line bundles with fixed numerical class is a smooth irreducible torsor under $\Pic^0(S)$.  On the other hand, each torsion element of $H^2(S,\mathbb{Z})$ produces a new connected component of this moduli space. Clearly, Proposition \ref{p0} gives the weak Brill-Noether property for all such connected components if $ab \neq 0$. In fact, this stronger result also holds for numerically trivial line bundles, as will be clear in the proof of Lemma \ref{lemma2wbn}.
\end{Rem}

The following Lemma shows that we can reduce the computation of the cohomology of the general line bundle in a component of $\Pic(S)$ to a computation on the canonical cover of $S$.

\begin{Lem} \label{l1}
Fix a connected component $M$ of the Picard scheme of $S$. If there exists $L$ in $M$ such that $\dim H^i(\tilde{X},\pi^{*}L) \leq \ord(\omega_S)-1$, then the general line bundle in $M$ has vanishing cohomology in degree $i$.
\end{Lem}

\begin{proof}
Since $\pi$ is a finite morphism, in particular affine, it has no higher pushforward and by construction we have $\pi_{*}\mathscr{O}_{\tilde{X}} \cong \bigoplus_{j = 1}^{\ord(\omega_S)} \omega_S^{\otimes j}$.
Therefore, a direct application of the projection formula and of Leray spectral sequence gives
\[
H^i(\tilde{X},\pi^{*}L) = \bigoplus_{j=1}^{\ord(\omega_S)} H^i(S,L \otimes \omega_S^{\otimes j}).
\]
Now, since $\omega_S \in \Pic^0(S)$, all line bundles $L \otimes \omega_S^{\otimes j}$ lie in the same connected component of the Picard scheme of $S$. If $\dim H^i(\tilde{X},\pi^{*}L) \leq \ord(\omega_S)-1$, then at least one of the summands in the right hand side is zero, hence the claim follows by upper semicontinuity of cohomology, as any connected component of the Picard scheme is clearly irreducible by smoothness.
\end{proof}

\begin{Lem} \label{lemma1wbn}
Let $S$ be a bielliptic surface. Then, the moduli space of line bundles of numerical class $aA_0$ satisfies the weak Brill-Noether property for all $a \in \mathbb{Z}$.
\end{Lem}

\begin{proof}
    Notice that by Serre duality and Corollary \ref{c1} we can reduce to consider the case $a > 0$, in which case the cohomology clearly vanishes in degree 2 for slope-stability reasons. Furthermore, since $A^2= 0$, we have $h^0(S,L) = h^1(S,L)$ for all line bundles $L$ of numerical class $aA_0$, thus the weak Brill-Noether property is verified if and only if $H^0(S,L) = 0$ for some $L$. This moduli space always contains the line bundles of the form $L \otimes f^{*}\alpha$, for $\alpha \in \Pic^0(A/G)$ and
    \[
    L  = \mathscr{O}_S\left(\left\lfloor\frac{a}{\ord(\omega_S)}\right\rfloor A + \sum_i a_iD_i\right)
    \]
    where $D_i$ denotes a reduced fiber of $g$ of multiplicity $m_i$, $0\leq a_i \leq m_i-1$ and
    \[
    a = \ord(\omega_S)\left\lfloor\frac{a}{\ord(\omega_S)}\right\rfloor + \sum_i \frac{\ord(\omega_S)}{m_i}a_i.
    \]
    Notice that we have a finite \'{e}tale quotient $\psi:A \times B \to \tilde{X}$, which is the identity in the split case, so that $\mathscr{O}_{\tilde{X}}$ is a direct summand of $\psi_{*}\mathscr{O}_{A \times B}$. Now, $(\pi\circ\psi)^{*}L$ is clearly the pullback of a positive line bundle $L' \in \Pic(B)$ under the projection $A\times B \to B$. Therefore, letting $p:A\to A/G$, we can compute, using Künneth formula and noticing that for general $\alpha$ one has $p^{*}\alpha \in \Pic^0(A) \smallsetminus \{\mathscr{O}_A\}$,
    \[
    H^0(\tilde{X},\pi^{*}(L\otimes f^*\alpha)) \hookrightarrow H^0(A \times B, p^{*}\alpha \boxtimes L') = 0.
    \]
    Therefore, the claim follows by Lemma \ref{l1}.
\end{proof}

\begin{Lem} \label{lemma2wbn}
Let $S$ be a bielliptic surface.
\begin{itemize}
    \item[$1)$] If $S$ is of type $1,2,3,5$, so that $H^2(S,\mathbb{Z})_{\tor} \neq 0$, then the moduli space of line bundles of numerical class $bB_0$ satisfies the weak Brill-Noether property for all $b \in \mathbb{Z}$.
    \item[$2)$] If $S$ is of type $4,6,7$, so that $H^2(S,\mathbb{Z})_{\tor} = 0$, then the weak Brill-Noether property fails exactly for the moduli spaces of line bundles of numerical class $bB_0$, with $\lvert b \rvert \geq \lambda_S$.
\end{itemize}
\end{Lem}

\begin{proof}
    Let us start from point (2). As in the previous Lemma we reduce to consider the case $b > 0$. As before, this implies $H^2(S,L) = 0$ and $h^0(S,L) = h^1(S,L)$ for any line bundle of numerical class $bB_0$, thus it is enough to compute $H^0(S,L)$. Assume first that $bB_0$ is an integer multiple of a fiber of $f$, that is, $\lambda_S \mid b$. Using Leray spectral sequence for $f$ we can thus compute, for any $\alpha \in \Pic^0(A/G)$,
    \[
    \dim H^0(S,\mathscr{O}_S(bB_0) \otimes f^{*}\alpha) = \dim H^0(A/G, \mathscr{M}),
    \]
    where $\mathscr{M}$ is line bundle of degree $\frac{b}{\lambda_S}$ on $A/G$, hence $\dim H^0(S,\mathscr{O}_S(bB_0) \otimes f^{*}\alpha) = \frac{b}{\lambda_S}$. Consequently, $H^0(S,\mathscr{O}_S(bB_0) \otimes f^{*}\alpha) \neq 0$ for all $b \geq \lambda_S$. If $H^2(S,\mathbb{Z})_{\tor} = 0$, then every line bundle of numerical class $bB_0$ has form $\mathscr{O}_S(bB_0) \otimes f^{*}\alpha$ for some $\alpha \in \Pic^0(A/G)$, hence we conclude that in this case the weak Brill-Noether property fails when $\lvert b \rvert \geq \lambda_S$.

    The case $\lvert b \rvert < \lambda_S$, which only arises for non-split bielliptic surfaces, is different. Indeed, in this case, 
    \[
    \pi^*\mathscr{O}_S(bB_0) = \mathscr{O}_{\tilde{X}}\left(\frac{\ord(\omega_S)b}{\lambda_S}B_{\tilde{X}}\right),
    \]
    therefore, using the fibration $\tilde{X} \to A/H$, we can compute
    \[
    \dim H^0(\tilde{X}, \pi^*(\mathscr{O}_S(bB_0) \otimes f^*\alpha)) \leq \ord(\omega_S) -1
    \]
    and thus the general line bundle of numerical class $bB_0$ has vanishing $H^0$ by Lemma \ref{l1}.

    The proof of point (2) is then complete and it remains to check the weak Brill-Noether property for line bundles of the form $\mathscr{O}_S(bB_0) \otimes f^{*}\alpha \otimes \mathscr{M}$, where $\alpha \in \Pic^0(A/G)$ and $c_1(\mathscr{M})$ is a non-zero element in $ H^2(S,\mathbb{Z})_{\tor}$, which exists if $S$ is of type $1,2,3,5$. Recall that $ H^2(S,\mathbb{Z})_{\tor}$ has been computed explicitly in \cite{torsion}, we can thus work case by case.

    \textbf{Type 1} In this case, the fibration $g:S \to \mathbb{P}^1$ has 4 multiple fibers, each with multiplicity 2. Denote their underlying reduced scheme by $D_1,\dots,D_4$. Then there are 3 non-zero torsion classes in $H^2(S,\mathbb{Z})$, corresponding to the non-isomorphic line bundles $\mathscr{O}_S(D_1 - D_i)$, for $2 \leq i \leq 4$. Since all multiple fibers have the same multiplicity, equal to 2, and the canonical cover of $X$ is $A \times B$ we have
    \[
    \pi^{*}\mathscr{O}_S(D_1 - D_i) \cong \mathscr{O}_A \boxtimes \mathscr{O}_B(P - Q_i),
    \]
    for distinct 2-torsion points $P,Q_i \in B$. Therefore,
    \[
    H^0(\tilde{X},\pi^{*}(\mathscr{O}_S(bB_0) \otimes \mathscr{O}_S(D_1 - D_i) \otimes f^{*}\alpha)) \cong H^0(A,L) \otimes H^0(B,\mathscr{O}_B(P-Q_i)) = 0,
    \]
    where $L$ is a line bundle of degree $2b$ on $A$.
    We conclude that for bielliptic surfaces of type 1 the weak Brill-Noether property holds also for line bundles of numerical class $bB_0$; more precisely, the generic line bundle in 3 of the 4 connected components of the moduli space of line bundles of numerical class $bB_0$ has no non-zero cohomology group, as $\chi(S,\mathscr{O}_S(bB_0)) = 0$.

    \textbf{Type 2} As in the previous case, $g:S \to \mathbb{P}^1$ has 4 multiple fibers with multiplicity 2, with associated reduced schemes $D_1,\dots,D_4$. The differences $D_1 - D_i$ induce 3 torsion classes in $H^2(S,\mathbb{Z})$, one of which is trivial while the other 2 are non-zero and coincide. The same computation as above, with the only difference that one needs to consider the composition of $\pi$ with the finite \'{e}tale quotient $\psi:A \times B \to \tilde{X}$, gives the weak Brill-Noether property for one of the 2 connected components of $\{ L \in \Pic(S): L \equiv_{\text{num}} bB_0\}$.

    \textbf{Type 3} In this case, $g:S \to \mathbb{P}^1$ has 2 multiple fibers with multiplicity 4 and one with multiplicity 2. We denote by $D_1,D_2$ the reduced multiple fibers with multiplicity 4 and by $D$ the one with multiplicity 2. Then $\mathscr{O}_S(D-2D_1)$ and $\mathscr{O}_S(D-2D_2)$ induce the same non-zero torsion element of $H^2(S,\mathbb{Z})$ and $\pi^{*}\mathscr{O}_S(D - 2D_1) \cong \mathscr{O}_A \boxtimes \mathscr{O}_B(P-Q)$ for two distinct points $P,Q \in B$ that are fixed by the action of $\mathbb{Z}/4\mathbb{Z}$ on $B$. Thus, a similar computation as above implies the weak Brill-Noether property for one of the 2 connected components of $\{ L \in \Pic(S): L \equiv_{\text{num}} bB_0\}$.

    \textbf{Type 5} Here $g:S \to \mathbb{P}^1$ has 3 multiple fibers, each with multiplicity 3 and there are 2 non-zero elements in $ H^2(S,\mathbb{Z})_{\tor}$, induced by $\mathscr{O}_S(D_1-D_i)$, $i=1,2$, where $D_i$ denote as above the reduced multiple fibers. Then the same computation as for type 1 gives the weak Brill-Noether property for 2 of the 3 connected components of $\{ L \in \Pic(S): L \equiv_{\text{num}} bB_0\}$.
\end{proof}

\begin{Rem} \label{littlerem}
    We remark that the computation in the proof of Lemma \ref{lemma1wbn} is not affected by tensoring with the line bundles inducing non-zero torsion elements of $H^2(S,\mathbb{Z})$. Therefore, \emph{all} irreducible components of the moduli space of line bundles with numerical class $aA_0$ satisfy the weak Brill-Noether property.
\end{Rem}

For sheaves of higher rank, understanding the weak Brill-Noether property is much harder. However, in the case of isotropic Mukai vectors, we have the following result.

\begin{Prop} \label{isotropic}
Let $S$ be a bielliptic surface and $H$ be a generic polarization with respect to a primitive isotropic Mukai vector $\mathbf{v}=(r,D,s)$, with $r \geq 1$ and $D^2 \neq 0$. Let $n$ be a positive integer such that $nl(\mathbf{v}) \mid \text{\emph{ord}}(\omega_S)$. Then the moduli space $M_H(n\mathbf{v})$ of $\mu_H$-semistable sheaves of Mukai vector $n\mathbf{v}$ satisfies the weak Brill-Noether property.
\end{Prop}

\begin{proof}
    By \cite[Thm. 1.1]{nuer}, the condition $nl(\mathbf{v}) \mid \ord(\omega_S)$ ensures that there exists a $\mu_H$-stable torsion-free sheaf $\mathscr{F}$ of Mukai vector $n\mathbf{v}$. By \cite[Lemma 7.3]{nuer}, the pullback $\pi^{*}\mathscr{F}$ of $\mathscr{F}$ to the canonical cover $\tilde{X}$ is either $\mu_{\pi^{*}H}$-stable, in which case $n=l(\mathbf{v}) = 1$, or it is of the form
    \[
    \pi^{*}\mathscr{F} = \bigoplus_{j= 0}^{d-1} (g^j)^{*}\mathscr{G},
    \]
    for a generator $g$ of $\mathbb{Z}/\ord(\omega_S)\mathbb{Z} \leq G$, $2 \leq d \leq \ord(\omega_S)$ such that $d \mid \text{ord}(\omega_S)$ and some $\mu_{\pi^{*}H}$-stable sheaf $\mathscr{G}$ such that the $(g^j)^{*}\mathscr{G}$ are distinct for $0 \leq j \leq d-1$. In particular, $\mathscr{G}$ is a simple semi-homogeneous sheaf on the abelian variety $\tilde{X}$, hence by \cite[Thm. 5.8]{Mukai} it is the pushforward of a line bundle $L$ along an isogeny $\psi: Y \to \tilde{X}$. In the first case, the same is true for $\pi^{*}\mathscr{F}$. The condition $D^2 \neq 0$ forces $L^2 \neq 0$. If $L^2 > 0$, then either $L$ is ample, so that $H^1(Y,L) = H^2(Y,L) = 0$, or the dual $L^{*}$ is ample, so that $H^0(Y,L) = H^1(Y,L) = 0$. On the other hand, if $L^2<0$, then $H^0(Y,L) = H^2(Y,L) = 0$. In all cases, using the Leray spectral sequence for $\psi$ and $\pi$, we deduce that the same cohomology vanishing holds for $\mathscr{F}$. We thus have proved that \emph{all} $\mu_H$-stable sheaves in $M_H(n\mathbf{v})$ have at most one non-zero cohomology group.
\end{proof}

\section{Ulrich bundles on bielliptic surfaces}

Let us start by showing how the Hirzebruch-Riemann-Roch Theorem together with Bogomolov inequality restricts the possible Mukai vectors of Ulrich bundles to a finite list for each rank $r\geq 1$.

\begin{Prop} \label{p1}
Let $E$ be a Ulrich bundle on a bielliptic surface $S$ polarized by an ample divisor of numerical class $H = aA_0 + bB_0$ and assume that $H$ is not divisible in $\Num(S)$. Then its Mukai vector has form 
\[
\mathbf{v}^{\Ulrich}(r,k) := (r,kaA_0 + (3r-k)bB_0,2rab)
\]
for some integers $r\geq 1$, $r \leq k \leq 2r$.
\end{Prop}

\begin{proof}
    Since $E$ is Ulrich, we have $\chi(E(-H) ) = \chi(E(-2H)) = 0$, hence the Hirzebruch-Riemann-Roch Theorem implies that
    \begin{equation} \label{eq1}
        \mathbf{v}(E) = (r,D,0)e^H
    \end{equation}
    where $D \in \Num(S)$ satisfies
    \[
    2D\cdot H = rH^2.
    \]
    Writing $D = xA_0 + yB_0$ and using the fact that $\Num(S)$ is isomorphic to a hyperbolic plane, this means
    \[
    xb + ya = rab,
    \]
    which is clearly solved by 
    \[
    \begin{cases} x=ja \\y = (r-j)b \end{cases}
    \]
    for $j \in \mathbb{Z}$. Conversely, if $(x_0,y_0)$ is any integer solution to this equation, then $x_0b +(y_0-rb)a = 0$, which forces $x_0$ to be an integer multiple of $a$, as $\gcd(a,b) = 1$ by assumption. Therefore, these are the only integer solutions. 
    
    Moreover, since Ulrich bundles are $\mu_H$-semistable, the Bogomolov inequality implies $D^2 \geq 0$, thus $0 \leq j \leq r$. The claim then follows from (\ref{eq1}) putting $k = j + r$.
\end{proof}

\begin{Rem} \label{divisible}
If we drop the assumption that $H$ is not divisible in $\Num(S)$, then the possible Mukai vectors of Ulrich bundles of fixed rank $r$ with respect to $H$ are still finitely many. Indeed, keeping the same notation of the previous proof, i.e. $c_1(E(-H)) = xA_0 + yB_0$, we have that $(x,y) \in \Num(S)$ must be a lattice point of the line $\{xb + ya = rab\}$ lying in the nef cone.

As for the existence, we remark that by \cite[Cor. 3.2]{beauville}, the existence of Ulrich bundles of rank $r$ with respect to $H$ implies the existence of Ulrich bundles of rank $2r$ with respect to any multiple of $H$.
\end{Rem}

\begin{Rem} \label{ulrichdual}
We remark that if $E$ is a Ulrich bundle with respect to a polarization $H$, then also $E^{*}\otimes \omega_S \otimes \mathscr{O}_S(3H)$ is Ulrich with respect to $H$. Indeed, by Serre duality, one has, for $1 \leq j \leq 2$
\[
H^{\bullet}(S,E^{*}\otimes \omega_S \otimes \mathscr{O}_S((3-j)H)) \cong H^{2-\bullet}(S,E \otimes \mathscr{O}_S((j-3)H))^{*} = 0.
\]
This means that the existence of Ulrich bundles of Mukai vector $\mathbf{v}^{\Ulrich}(r,k)$ is equivalent to the existence of Ulrich bundles of Mukai vector $\mathbf{v}^{\Ulrich}(r,3r-k)$.
\end{Rem}

\begin{Rem} \label{ulrichlinebundles}
For $r=1$, we obtain that a Ulrich line bundle with respect to $H = aA_0 + bB_0$, if it exists, has numerical class $aA_0 + 2bB_0$ or $2aA_0 + bB_0$. By Remark \ref{ulrichdual}, the two cases are equivalent to consider. Given a divisor $D$ of numerical class $aA_0 + 2bB_0$, we have $\chi(S,\mathscr{O}_S(D-H)) = \chi(S,\mathscr{O}_S(D-2H)) = 0$, thus a Ulrich line bundle of numerical class $aA_0 + 2bB_0$ exists if and only if there exists an irreducible component of $
\{ L \in \Pic(S): L \equiv_{\text{num}} bB_0\}$
which satisfies the weak Brill-Noether property and is mapped by $- \otimes \mathscr{O}_S(-H)$ to an irreducible component of $\{ L \in \Pic(S): L \equiv_{\text{num}} -aA_0\}$ that also satisfies the weak Brill-Noether property. Under the assumption $b \geq \lambda_S$, this happens if and only if $S$ is of type $1,2,3,5$ by Proposition \ref{wbnlb} and Remark \ref{littlerem}.
\end{Rem}

\begin{Prop} \label{special}
Every bielliptic surface polarized by an ample divisor $H$ carries a $\mu_H$-stable Ulrich bundle of Mukai vector $\mathbf{v}^{\Ulrich}(2,3) = (2,3H,2H^2)$.
\end{Prop}

\begin{proof}
Such Ulrich bundles were originally constructed by Beauville (see e.g. \cite{beauville}, \cite{casnati}) using the Cayley-Bacharach property on zero-dimensional subschemes of $X$. In view of the existence result \cite[Thm. 1.1]{nuer}, let us give here a different proof, inspired by an argument of \cite{Nijsse}. 

Let $E$ be a $\mu_H$-stable vector bundle in $M_H(2,3H,2H^2)$, which exists by \cite[Thm. 1.1]{nuer}. Then $\mu_H(E(-H)) > 0$, hence by stability and Serre duality we have $H^2(S,E(-H)) = 0$. Similarly $H^0(S,E(-2H)) = 0$. A direct computation gives $\chi(S,E(-H)) = \chi(S,E(-2H)) = 0$ and $\mathbf{v}(E(-H)) = \mathbf{v}(E^{*}(2H) \otimes \omega_S) = (2,H,0)$, thus, using Serre duality, the claim reduces to prove that the general $\mu_H$-stable sheaf in $M_H(2,H,0)$ has no non-zero global sections.

Define 
\[
W = \{\mathscr{F} \in M_H^{\mu s}(2,H,0): H^0(S,\mathscr{F}) \neq 0\}
\]
and let $\mathscr{F} \in W$. If we take a non-zero global section $s$ of $\mathscr{F}$, $\text{coker}(s)$ is a torsion-free sheaf of rank 1, thus $\mathscr{F}$ sits in a short exact sequence
\[
\begin{tikzcd}
0 \arrow[r] & \mathscr{O}_S \arrow[r] & \mathscr{F} \arrow[r] & \mathscr{I}_Z (H) \arrow[r] & 0,
\end{tikzcd}
\]
where $Z$ is a zero-dimensional subscheme of $S$ of length $\frac{H^2}{2}$. Thus, if we let 
\[
V = \{(\mathscr{F},s):\mathscr{F}\in W,s\in \mathbb{P}H^0(S,\mathscr{F})\}
\]
we have a morphism $\varphi: V\to \text{Hilb}^{\frac{H^2}{2}}(S)$ whose fiber is
\[
\varphi^{-1}(Z) = \mathbb{P}\text{Ext}^1(\mathscr{I}_Z (H),\mathscr{O}_S).
\]
Using Serre duality and the fact that $\chi(\mathscr{I}_Z (H+K_S))= 0$, where $K_S$ is a canonical divisor, we see that the fiber of $\varphi$ over $Z$ is non-empty if and only if $h^0(S,\mathscr{I}_Z (H+K_S)) \geq 1$. Therefore, $\text{im}(\varphi) \subseteq \bigcup_{i\geq 0} \Delta_i$, where
\[
\Delta_i = \{Z \in \text{Hilb}^{\frac{H^2}{2}}(S):h^0(S,\mathscr{I}_Z (H+K_S)) = 1 + i\}.
\]
Now consider the incidence variety
\[
T = \{(C,Z)\in \vert H+K_S \vert \times \text{Hilb}^{\frac{H^2}{2}}(S): Z \subseteq C\}
\]
with its projections $\pi_1$ and $\pi_2$ to the linear system of $H+K_S$ and $\text{Hilb}^{\frac{H^2}{2}}(S)$ respectively. Clearly, $Z \subseteq C$ for a curve $C$ cut out by $s \in H^0(S,\mathscr{O}_S(H+K_S))$ if and only if $s$ lies in the kernel of the restriction map $H^0(S,\mathscr{O}_S(H+K_S)) \to H^0(S,\mathscr{O}_Z(H+K_S))$, which is $H^0(S,\mathscr{I}_Z(H+K_S))$. Therefore for all $i \geq 0$ we have
\[
\frac{H^2}{2} + \frac{H^2}{2} - 1 = \text{length}(Z) + \dim \vert H+K_S \vert \geq \dim T \geq \dim \pi_2^{-1}(\Delta_i) = \dim \Delta_i + i,
\]
so, $\dim \varphi^{-1}(\Delta_i) \leq \dim\Delta_i + \dim \varphi^{-1}(Z) \leq H^2 -i -1  + i  = H^2  -1$. Using the obvious forgetful morphism $V \to W$, we thus conclude that $\dim W \leq \sup_{i \geq 1}\dim \varphi^{-1}(\Delta_i) \leq H^2 -1$. On the other hand,
\[
\dim_{[\mathscr{F}]}M_H^{\mu s}(2,H,0)  \geq \dim\Ext^1(\mathscr{F},\mathscr{F}) - \dim\Ext^2(\mathscr{F},\mathscr{F})= H^2 + 1,
\]
thus the general sheaf $\mathscr{F}$ in $M_H^{\mu s}(2,H,0)$ has $H^0(S,\mathscr{F})= 0$ and we are done.
\end{proof}

\begin{proof}[Proof of Thm. \ref{main}]
    First of all, notice that all Mukai vectors $\mathbf{v}^{\Ulrich}(r,k)$, for $r\geq 1$, $r \leq k \leq 2r$, satisfy Bogomolov inequality, thus the moduli spaces $M_H(\mathbf{v}^{\Ulrich}(r,k))$ are non-empty by \cite[Thm. 1.1]{nuer}. It is then enough to check the vanishing cohomology condition.

    It follows from Proposition \ref{p1} and Remark \ref{ulrichlinebundles} that $S$ carries Ulrich line bundles with respect to $H$ if and only if $S$ is of type $1,2,3,5$. Since direct sums of Ulrich bundles are Ulrich, a simple induction argument shows that in this case $S$ carries Ulrich bundles of Mukai vector $\mathbf{v}^{\Ulrich}(r,k)$ for all $r\geq 1$, $r \leq k \leq 2r$.
    
    Let $S$ be of type 7. Then, by Lemma \ref{lem:intermediate bielliptic order composite}, there exist bielliptic surfaces $\tilde{S}$, $\tilde{S}'$ of type 1, respectively 5, that are \'{e}tale coverings $\tilde{\pi}$, $\tilde{\pi}'$  of $S$ of degree $3$, respectively 2. Since $\tilde{S}$ and $\tilde{S}'$ carry Ulrich line bundles with respect to the pullback of $H$, then, using that
    \[
    H^i(S, \tilde{\pi}_{*}L \otimes \mathscr{O}_S(jH)) \cong H^i(\tilde{S},L \otimes \tilde{\pi}^{*}\mathscr{O}_{S}(jH)),
    \]
    we deduce that $S$ carries Ulrich bundles of Mukai vectors $\mathbf{v}^{\Ulrich}(2,2)$ and $\mathbf{v}^{\Ulrich}(3,3)$. Similarly, using Lemma \ref{lem:intermediate bielliptic lambda bigger than 1}, we obtain that if $S$ is of type 4, then it carries Ulrich bundles of Mukai vector $\mathbf{v}^{\Ulrich}(2,2)$ and if $S$ is of type 6, then it carries Ulrich bundles of Mukai vector $\mathbf{v}^{\Ulrich}(3,3)$.

    In contrast, if $S$ is of type $4$, then $S$ does not carry Ulrich bundles of Mukai vector $\mathbf{v}^{\Ulrich}(3,3)$. Indeed, the Mukai vector $\mathbf{v}^{\Ulrich}(3,3)$ is isotropic and divisible by 3 in the algebraic Mukai lattice $\Hal(S,\mathbb{Z})$, thus $M_H^{\mu s}(\mathbf{v}^{\Ulrich}(3,3))=\varnothing$ by \cite[Prop. 4.1]{nuer}. But then, since Jordan-H\"older factors of Ulrich bundles are clearly Ulrich, the existence of Ulrich bundles in $M_H(\mathbf{v}^{\Ulrich}(3,3))$ would imply the existence of Ulrich line bundles on $S$, contradicting $H^2(X,\mathbb{Z})_{\tor} = 0$. Similarly, a bielliptic surface of type 6 does not carry Ulrich bundles of Mukai vector $\mathbf{v}^{\Ulrich}(2,2)$.

    We claim that bielliptic surfaces of type 4 carry Ulrich bundles of Mukai vector $\mathbf{v}^{\Ulrich}(3,4)$, or $\mathbf{v}^{\Ulrich}(3,5)$, which is equivalent by Remark \ref{ulrichdual}. Together with Proposition \ref{special}, this concludes the proof of the Theorem. 
    
    We want to show that there exists a vector bundle $\mathscr{F} \in M_H(\mathbf{v}^{\Ulrich}(3,4))$ such that 
    \[
    H^{\bullet}(S,\mathscr{F}(-H)) = H^{\bullet}(S,\mathscr{F}(-2H)) = 0.
    \]
    Let $E$ be a Ulrich bundle of Mukai vector $\mathbf{v}^{\Ulrich}(2,2)$. The vector bundle $\mathscr{O}_S(aA_0) \oplus E(-H)$ is $\mu_H$-semistable since direct sum of $\mu_H$-semistable bundles with the same slope. This gives, after tensoring with a general element of $\Pic^0(S)$, a $\mu_H$-semistable vector bundle of Mukai vector $\mathbf{v}(\mathscr{F}(-H))$ with vanishing cohomology in all degrees, thanks to Lemma \ref{lemma1wbn}. Analogously, if $E'$ is a Ulrich bundle of Mukai vector $\mathbf{v}^{\Ulrich}(2,3)$, as constructed in Proposition \ref{special}, then the direct sum $\mathscr{O}_S(-aA_0) \oplus E'(-2H)$ is a $\mu_H$-semistable vector bundle of Mukai vector $\mathbf{v}(\mathscr{F}(-2H))$. After tensoring with a general element of $\Pic^0(S)$, this gives, by Lemma \ref{lemma1wbn}, a vector bundle in $M_H(\mathbf{v}(\mathscr{F}(-2H)))$ with vanishing cohomology in all degrees. By upper semicontinuity of cohomology it is then enough to notice that the moduli space $M_H(\mathbf{v}^{\Ulrich}(3,4))$ is in this case irreducible, which follows directly from Lemma \ref{irredcomp}.
\end{proof}

\section{Degree of irrationality of bielliptic surfaces}

Our strategy for computing $\irr(S)$ is to reduce this problem to the question of the existence of stable rank 2 vector bundles on $S$ with certain Chern classes. This strategy was first introduced by Moretti \cite{moretti} in the case of K3 surfaces and is based on the observation that, if $\varphi_V:S \dashrightarrow \mathbb{P}^2$ is a generically finite rational map induced by some $V \in \text{Gr}(3,H^0(S,L))$ with base locus $\text{Bs}(V)$, then the dual of the kernel bundle $\ker (\text{ev}:V \otimes \mathscr{O}_S \to L)$ is
\[
E \cong \varphi_V^*(T\mathbb{P}^2 \otimes \mathscr{O}_{\mathbb{P}^2}(-1)),
\]
and for a section $s \in H^0(\mathbb{P}^2,T\mathbb{P}^2 \otimes \mathscr{O}_{\mathbb{P}^2}(-1))$, the fiber of $\varphi|_{S \smallsetminus \text{Bs}(V)}$ over the point $p = Z(s) \in \mathbb{P}^2$ is exactly $Z(\varphi_V^*s) \cap (S\smallsetminus \text{Bs}(V))$. More precisely, the strategy in \cite{moretti} is based on the following definition. 

\begin{Def}[{\cite[Def. 1.3]{moretti}}] \label{goodpair}
    Given a smooth projective variety of dimension $n$, a pair $(E,V)$ with $E$ a reflexive sheaf on $X$ and $V \in \text{Gr}(n+1,H^0(X,E))$ is called a \emph{good pair} if $V$ generates $E$ in codimension 2 and $H^0(X,E^*) = 0$.
\end{Def}

Given a good pair $(E,V)$, since the evaluation map $\text{ev}_E:V \otimes \mathscr{O}_X \to E$ is surjective in codimension 2, we get that $c_1(\text{im}(\text{ev}_E)) = c_1(E)$ and $\text{im}(\text{ev}_E)^* \cong E^*$, therefore the kernel $\ker (\text{ev}_E)$ is the dual of $L = \det E$. Thus, dualizing the evaluation map we obtain an exact sequence
\[
\begin{tikzcd}
0 \arrow[r] & E^{*} \arrow[r] & V^{*} \otimes \mathscr{O}_X \arrow[r] & L,
\end{tikzcd}
\]
which identifies $V^*$ with a subspace of $H^0(X,L)$, as $H^0(X,E^*) = 0$. We can thus associate to $E$ a rational map $\varphi_V:X \dashrightarrow \mathbb{P}(V)$ with base locus given by the union of the locus where $V$ does not generate $E$ and the locus where $E$ is not locally free. The following result that we are going to apply is the combination of the general \cite[Prop. 1.5]{moretti} with its application to surfaces \cite[Cor. 1.7]{moretti}.

\begin{Thm} \label{morettithm}
Given a good pair $(E,V)$ on a smooth projective surface $S$, with associated generically finite rational map $\varphi_V:X \dashrightarrow\mathbb{P}^2$, one has
\[
\deg(\varphi_V) = c_2(E) - \deg\left(\bigcap_{s \in V} Z_{\text{\emph{cycle}}}(s)\right).
\]
\end{Thm}

\begin{Lem} \label{deg12}
Let $S$ be a bielliptic surface. Then for all ample line bundles $L = \mathscr{O}_S(D)$ such that $D^2 = 12$ and all $V \in \text{\emph{Gr}}(3,H^0(S,L))$, the rational map $\varphi_V:S \dashrightarrow \mathbb{P}^2$ is generically finite.
\end{Lem}

\begin{proof}
    First of all, notice that $\varphi_V$ is the composition
    \[
    S \dashrightarrow \mathbb{P}H^0(S,L)^* \dashrightarrow \mathbb{P}V^*,
    \]
    where the first map is the rational map induced by the complete linear system $\lvert L\rvert$ and the second one is a projection $\mathbb{P}^5 \dashrightarrow\mathbb{P}^2$. Now, by Proposition \ref{p0}, any effective divisor $E \equiv_{\text{num}} xA_0 + yB_0$ on $S$ has $x\geq 0$, $y \geq 0$ and we have already seen, as noted in Remark \ref{littlerem}, that the only line bundle of numerical class zero with non-zero space of global sections is $\mathscr{O}_S$. This implies that if $E$ is a non-zero effective divisor on $S$, we have $(D-E)^2 < D^2$, therefore
    \[
    \dim H^0(S,\mathscr{O}_S(D-E)) < \dim H^0(S,\mathscr{O}_S(D)),
    \]
    as ample line bundles on $S$ have no higher cohomology. We deduce that the linear system $\lvert L\rvert$ has no fixed components. Thus $S \dashrightarrow \mathbb{P}H^0(S,L)^*$ is defined away from a finite set of points of $S$ and consequently the same holds for $\varphi_V$; furthermore $\varphi_V^*\mathscr{O}_{\mathbb{P}^2}(1)$ extends to the line bundle $L$. Suppose by contradiction that $\varphi_V$ is not generically finite. Then it induces a dominant rational map to an irreducible plane curve $C$. Blowing up the indeterminacy locus, we obtain a morphism $\tilde{\varphi_V}:\tilde{S} \to C$. Let $\nu:C'\to C$ be the normalization of $C$. By the universal property of normalization, $\tilde{\varphi_V}$ factors through a surjective morphism $\psi:\tilde{S} \to C'$, inducing a surjective morphism from $\text{Alb}(\tilde{S})$ to the Jacobian of $C'$. Since the irregularity of $S$ is one, $C'$ has to be either an elliptic curve $E$ or $\mathbb{P}^1$.
    
    We start by dicarding the second case. If $C' = \mathbb{P}^1$, then the pullback of $L$ under the blow-up morphism $\text{Bl}:\tilde{S} \to S$ is
    \[
    \text{Bl}^*L \cong \tilde{\varphi_V}^*(\mathscr{O}_{\mathbb{P}^2}(1)|_C) \cong \psi^*\nu^*(\mathscr{O}_{\mathbb{P}^2}(1)|_C) \cong \psi^*\mathscr{O}_{\mathbb{P}^1}(d)
    \]
    for some $d \geq 2$, thus it is divisible in $\Num(\tilde{S})$. But then, it is also divisible in $\Num(S)$, since $\Num(\tilde{S}) =\bigoplus_i \mathbb{Z}E_i \oplus \Num(S)$, where $E_i$ denote the exceptional divisors of the blow-up. This is a contradiction since every line bundle of degree 12 is non-divisible in $\Num(S)$. Indeed, if we write $D = nD'$ for some positive integer $n$ and $D' \in \Num(S)$, then $D^2 = 12$ forces $n=2$ and $(D')^2=3$, contradicting the fact that $\Num(S)$ is an even lattice.
    
    Finally, in case $C'$ is an elliptic curve $E$, we use the universal property of the Albanese map to obtain the commutative square
    \[
    \begin{tikzcd}
\tilde{S} \arrow[r] \arrow[d] \arrow[dr]
& S \arrow[d, dashrightarrow] \\
\text{Alb}(\tilde{S}) \arrow[r]
& E
\end{tikzcd}
    \]
    The Albanese morphism of $\tilde{S}$ clearly contracts the exceptional divisors, thus it restricts to the Albanese morphism of $S$, that is, the rational map $\varphi_V$ factors through $f:S \to A/G$. This clearly contradicts the fact that $L$ is ample, by the characterization of the ample cone of $S$.
\end{proof}

\begin{proof}[Proof of Thm. \ref{main2}]
Let $D$ be a divisor of numerical class $3A_0 + 2B_0$ and let $E$ be a locally free sheaf in $M_H^{\mu s}(\mathbf{v})$, with $\mathbf{v} = (2,D,3)$; such $E$ exists by \cite[Thm. 1.1]{nuer}. We have $H^2(S,E) =0$ since $E$ is $\mu_H$-stable with positive slope. Therefore $H^0(S,E^*) = 0$ and, as $\mathbf{v}(E)$ is isotropic, $\dim H^0(S,E) = \chi(S,E) = 3$ by Proposition \ref{isotropic}. We claim that $(E,H^0(S,E))$ is a good pair in the sense of Definition \ref{goodpair}, that is, $E$ is globally generated in codimension 2. 

The choice of the numerical class of $D$ implies that $l(\mathbf{v}) = 1$, so that $\pi^*E$ is a simple semi-homogeneous vector bundle on $\tilde{X}$ such that $c_1(\pi^*E) = 3A_{\tilde{X}} + \frac{2\ord(\omega_S)}{\lambda_S}B_{\tilde{X}}$, with notation as in Section 2. By the criterion for global generation of vector bundles on abelian varieties proved by Pareschi \cite[Thm. 2.1]{pareschi} we have that $\pi^*E$ is globally generated. Indeed, it is enough to show that there exists an ample line bundle $L$ on $\tilde{X}$ such that $$H^{i}(\tilde{X},\pi^*E \otimes L^{-1} \otimes \alpha) = 0$$ for all $\alpha \in \Pic^0(\tilde{X})$ and $i \geq 1$. This is easily achieved, for instance with $L = \mathscr{O}_{\tilde{X}}(A_{\tilde{X}} + B_{\tilde{X}})$, as in the proof of Proposition \ref{isotropic}, since $\pi^*E \otimes L^{-1} \otimes \alpha$ is the pushforward of a positive line bundle under an isogeny. Notice here that the fact that $L$ is ample follows classically from $L^2 = 2\lambda_S > 0$, together with $H^0(\tilde{X},L) \neq 0$, using Hodge index theorem and the fact that there is no curve of negative self-intersection on $\tilde{X}$.

The global generation of $\pi^*E$ is equivalent to the surjectivity of the joint evaluation map
\[
\eta: \bigoplus_{j=1}^{\ord(\omega_S)} \left(H^0(S,E \otimes \omega_S^{\otimes j}) \otimes \omega_S^{\otimes(\ord(\omega_S)-j)}\right) \to E
\]
obtained by twisting the evaluation map for $E \otimes \omega_S^{\otimes j}$ by $\omega_S^{\otimes (\ord(\omega_S)-j)}$ for $1 \leq j \leq \ord(\omega_S)$ and then summing over $j$. Notice here that $\dim H^0(S,E \otimes \omega_S^{\otimes j}) = 3$ for all $j$ by Proposition \ref{isotropic}. This implies that $E$, as well as $E \otimes \omega_S^{\otimes j}$ for $1 \leq j < \ord(\omega_S)$, is generically globally generated\footnote{This can also be deduced from the stability of $E$. Indeed, if $\text{im}(\text{ev}_E)$ has generic rank 1, then it has form $\mathscr{I}_Z(L)$ for a zero-dimensional subscheme $Z$ of $S$ and a line bundle $L$ such that $\dim H^0(S,L) \geq 3$. Using the computations of the cohomology of line bundles of section 3, it is easy to show that this forces $\mu_H(\mathscr{I}_Z(L)) = \mu_H(L) \geq \mu_H(E)$ for all ample divisors $H = aA_0 + bB_0$ such that $\frac{b}{2} \leq a \leq 2b$, contradicting the slope stability of $E$.}, that is, the cokernel of its evaluation map $\text{ev}_E$ is a torsion sheaf $T$. Furthermore, assuming by contradiction that the support of $T$ contains a divisor, we have $$\det E \neq \det (\text{im(ev}_E)) = \det (\text{im((ev}_{E\otimes \omega_S^{j}}) \otimes \omega_S^{\ord(\omega_S)-j})).$$ But then $\det (\text{im}(\eta)) \neq \det E$, contradicting the surjectivity of $\eta$. Therefore $E$ is globally generated in codimension 2.

We thus can conclude that $(E,H^0(S,E))$ is good pair and that the dual of the evaluation map identifies a subspace $V \in \text{Gr}(3,H^0(S,\mathscr{O}_S(D)))$. By Lemma \ref{deg12}, the induced rational map $\varphi_V:S \dashrightarrow\mathbb{P}^2$ is generically finite so Theorem \ref{morettithm} gives $\deg(\varphi_V) \leq c_2(E) = 3$.
\end{proof}

\begin{Rem}
Our choice of $E$ is minimal in the following sense. In the proof of Theorem \ref{main2} we need to start with a stable vector bundle $E$ of Mukai vector $(2,D,s)$ such that $s = \chi(S,E) \geq 3$, in order to ensure that $\dim H^0(S,E) \geq 3$. On the other hand, Bogomolov inequality imposes $D^2 - 4s \geq 0$. Hence, the minimal value of $c_2(E) = \frac{D^2}{2} - s$ is obtained for $D^2 = 12$.
\end{Rem}

\subsection*{Acknowledgements}
This paper is part of the author's Ph.D. project, which is supported by the project grant n. VR2023-03837. I would like to thank my Ph.D. supervisor Sofia Tirabassi for her guidance and many helpful discussions.

\bibliographystyle{plain}
{\footnotesize
\bibliography{biblio_paper.bib}}

\end{document}